\documentclass[10pt]{amsart}
\usepackage{amsbsy}
\usepackage{amsfonts,amssymb,graphicx,amsmath,amsthm, float, color}
\usepackage{fullpage}
\theoremstyle{plain}

\newtheorem{theorem}{Theorem}

\newtheorem{lemma}[theorem]{Lemma}
\newtheorem{prop}[theorem]{Proposition}
\newtheorem{remark}[theorem]{Remark}

\theoremstyle{definition}
\newtheorem{definition}[theorem]{Definition}

\newcommand{\Z}{\mathbb{Z}}
\newcommand{\diam}{{\rm{diam}}}

\newcommand{\lcm}{{\rm{lcm}}}

\newtheoremstyle{indented}{3pt}{3pt}{\addtolength{\leftskip}{2.5em}}{}{\bfseries}{.}{.5em}{}
\theoremstyle{indented}

\title{Radio Number of Hamming Graphs of Diameter 3}
\author{Jason DeVito}
\author{Amanda Niedzialomski}
\author{Jennifer Warren}

\makeindex
\begin{document}
\maketitle

\begin{abstract}
For $G$ a simple, connected graph, a vertex labeling $f:V(G)\to \Z_+$ is called a \emph{radio labeling of $G$} if it satisfies $|f(u)-f(v)|\geq\diam(G)+1-d(u,v)$ for all distinct vertices $u,v\in V(G)$.  The \emph{radio number of $G$} is the minimal span over all radio labelings of $G$.  If a bijective radio labeling onto $\{1,2,\dots,|V(G)|\}$ exists, $G$ is called a \emph{radio graceful} graph.  We determine the radio number of all diameter 3 Hamming graphs and show that an infinite subset of them is radio graceful.
\end{abstract}

\section{Introduction}
In this paper we compute radio numbers of the Hamming graphs $K_{\ell}\square K_m\square K_n$, where $\ell,m,n\geq 2$ and $K_n$ denotes the complete graph with $n$ vertices. We show that these graphs are radio graceful unless $\ell=m=2$ or $(\ell,m,n)=(2,3,3)$. This produces an infinite family of non-trivial radio graceful graphs. The first such families were given in \cite{N16}, where the third named author considers the Hamming graphs of the form $K_{n_1}\square K_{n_2}\square\cdots\square K_{n_d}$ where $n_1=n_2=\dots=n_d$ and, separately, where $n_1,n_2,\dots, n_d$ are pairwise relatively prime, and constructs consecutive radio labelings for these types of graphs in certain cases. We will use a similar technique to define consecutive radio labelings for all radio graceful Hamming graphs of diameter 3.

For a simple, connected graph $G$ with vertex set $V(G)$ and a positive integer $k$, we call a vertex labeling $f:V(G)\to\Z_+$ a $k$-radio labeling if it satisfies
\begin{equation*}
|f(u)-f(v)|\geq k+1-d(u,v)
\end{equation*}
for all distinct $u,v\in V(G)$.  This definition, given in 2001 in \cite{Chartrand}, encompasses some previously defined labelings, including vertex coloring, which is equivalent to $1$-radio labeling.  Other $k$-radio labelings that are studied include $L(2,1)$-labeling ($2$-radio labeling, \cite{GR92}), $L(3,2,1)$-labeling ($3$-radio labeling, \cite{321}), and radio labeling ($k$-radio labeling with maximum\footnote{The $k$ value of a $k$-radio labeling has a natural relationship to distance in $G$.  Namely, vertices of distance $k$ apart or less must have distinct images under $f$.  We therefore consider $k\leq\diam(G)$.} $k$ of the diameter of $G$, \cite{CEHZ}).  These labelings have historical ties to the problem of optimally assigning radio frequencies to transmitters in order to avoid interference between transmitters.  This so-called Channel Assignment Problem was framed as a graph labeling problem by Hale in 1980 in \cite{Hale}, by modeling transmitters and their frequency assignments with vertices of a graph and a labeling of them.  The relevance of $k$-radio labeling to this problem is clear; a pair of vertices (transmitters) with a relatively small distance must have a relatively large difference in labels (frequencies).  The original application is no longer the central motivation for studying $k$-radio labeling, and we do not limit our scope to only graphs relevant to this model. In this paper, we work within the framework of radio labeling, where $k=\diam(G)$.

\begin{definition}
Let $G$ be a simple, connected graph.  A vertex labeling $f:V(G)\to\Z_+$ is a \emph{radio labeling of $G$} if it satisfies
\begin{equation}\label{introrc}
|f(u)-f(v)|\geq\diam(G)+1-d(u,v)
\end{equation}
for all distinct $u,v\in V(G)$.
\end{definition}

Inequality \eqref{introrc} is called the \emph{radio condition}.  The largest element in the range of $f$ is called the \emph{span of $f$}.

\begin{definition}
Let $G$ be a simple, connected graph.  The minimal span over all radio labelings of $G$ is the \emph{radio number of $G$}, denoted ${\textrm{rn}}(G)$.
\end{definition}

\begin{remark}
We use the codomain of $\Z_+$ for radio labeling (and $k$-radio labeling), while some authors use a codomain the $\Z_+\cup\{0\}$.  Radio numbers and labelings are converted from one convention to the other by a shift of 1.
\end{remark}

\begin{remark}
If a graph $G$ has diameter 3, the definitions for a radio labeling of $G$ and an $L(3,2,1)$-labeling of $G$ are identical; the radio number of $G$ and the analogous $L(3,2,1)$-number of $G$ are equal.  This is the case for the graphs we consider in this paper, so the results are relevant to both labelings.\end{remark}

Unlike $k$-radio labeling with $k<\diam(G)$, radio labeling is an injective labeling, and therefore ${\textrm{rn}}(G)\geq|V(G)|$.  We are interested in graphs $G$ for which ${\textrm{rn}}(G)=|V(G)|$, which occurs when there exists a radio labeling $f$ with image $\{1,2,\dots,|V(G)|\}$.  We call these graphs \emph{radio graceful}, first named in \cite{SR}.

\begin{definition}
A radio labeling $f$ of a graph $G$ is a \emph{consecutive radio labeling of $G$} if $f(V(G))=\{1,2,\dots,|V(G)|\}$.  A graph for which a consecutive radio labeling exists is called \emph{radio graceful}. 
\end{definition}

The complete graphs $K_n$ are trivially radio graceful; as $\diam(K_n)\leq1$, the radio condition is satisfied for any injective vertex labeling of $K_n$.  Then any vertex labeling that maps $V(K_n)$ onto $\{1,2,\dots, n\}$ is automatically a consecutive radio labeling of $K_n$.  Radio graceful graphs with diameter larger than 1 are nontrivial examples.  The higher the diameter of a graph, the more restrictive the requirement to have an image of consecutive integers is  (see Proposition \ref{prelims_rgc_prop}). Examples of radio graceful graphs are sought, and in our study here of Hamming graphs of diameter 3, an infinite family of examples is found. More precisely, the main results of this paper state the following.

\begin{theorem} Suppose $2\leq \ell\leq m\leq n$.  Then the Hamming graph $K_{\ell}\square K_m\square K_n$ is radio graceful unless $\ell = m = 2$ or $(\ell,m,n) = (2,3,3)$.
\end{theorem}

In the exceptional cases, we explicitly compute the radio numbers.

\begin{theorem}  The radio number of $K_2\square K_2\square K_n$ is $6n-1 $ and the radio number of $K_2\square K_3\square K_3$ is $20 $.
\end{theorem}

Hamming graphs and other Cartesian graph products have been fruitful areas of study in the $k$-radio labeling context and have been particularly useful for finding examples of radio graceful graphs.  For $k$-radio labeling results involving Cartesian graph products, see \cite{JLV}-\cite{MTWY}, \cite{N16}, and \cite{MC}.
\section{Preliminaries}
Graphs are assumed simple and connected.  We denote the distance between vertices $u$ and $v$ in a graph $G$ by $d_G(u,v)$, or, if $G$ is clear from context, by $d(u,v)$.  We use the convention that $a\ ({\rm{mod}}\ n)\in\{1,2,\dots,n\}$ throughout.

We call an ordered list of the vertices of $G$ an \emph{ordering} if it is in one-to-one correspondence with $V(G)$. If $f$ is a consecutive radio labeling of $G$, then there is an ordering $x_1,x_2,\dots,x_n$ of $V(G)$ such that $f(x_i)=i$ for all $i\in\{1,2,\dots,n\}$. The ordering contains all of the information about the consecutive radio labeling. In light of this, the next proposition follows easily from the radio condition \eqref{introrc}.

\begin{prop}\label{prelims_rgc_prop}
A graph $G$ is radio graceful if and only if there exists an ordering $x_1,x_2,\dots,x_n$ of the vertices of $G$ such that 
\begin{equation}\label{prelims_rgc}
d(x_i,x_{i+\Delta})\geq\diam(G)-\Delta+1
\end{equation}
for all $\Delta\in\{1,2,\dots,\diam(G)-1\}$, $i\in\{1,2,\dots,n-\Delta\}$.
\end{prop}

The inequality \eqref{prelims_rgc} is called the \emph{radio graceful condition}.  As diameter increases, the radio graceful condition must be satisfied for more values of $\Delta$, which underlines the difficulty of finding examples of radio graceful graphs of higher diameter.

\begin{definition}\label{prelims_cart_prod_def}
The \emph{Cartesian product of graphs $G$ and $H$}, denoted $G\square H$, has the vertex set $V(G)\times V(H)$ and has the edges defined by the following property.  Vertices $(u,v),(u',v') \in V(G\square H)$ are adjacent if 
\begin{enumerate}
\item[1.] $u=u'$ and $v$ is adjacent to $v'$ in $H$, or
\item[2.] $v=v'$ and $u$ is adjacent to $u'$ in $G$.
\end{enumerate}
\end{definition}

The distance and diameter are inherited nicely from the factor graphs:
\begin{equation*}
d_{G\square H}((u,v),(u',v'))=d_G(u,u')+d_H(v,v').
\end{equation*}

\section{Radio graceful $K_\ell\square K_m\square K_n$}
In this section we show that the Hamming graphs $K_\ell\square K_m\square K_n$ with $\ell,m,n\geq 2$ and $(\ell,m,n)\not\in\{(2,3,3)\}\cup\{(2,2,n)\colon n\in\mathbb{N}\}$ are radio graceful.
First, we define a list of vertices of $K_\ell\square K_m\square K_n$; then we prove that this list is in one-to-one correspondence with $V(K_\ell\square K_m\square K_n)$, confirming that the list is indeed an ordering of the vertices; next we show that this ordering satisfies the consecutive radio condition, which proves our desired result.
\begin{subsection}{Definition of the ordering $x_{1}, x_{2},...,x_{\ell mn}$}\label{definition}
Consider a Hamming graph $K_{\ell}\square K_{m}\square K_{n}$ with $\ell,m,n\geq 2$ 
 and denote $V(K_{\ell})=\{u_{1},u_{2},...,u_{\ell}\}$, $V(K_{m})=\{v_{1},v_{2},...,v_{m}\}$, and $V(K_{n})=\{w_{1},w_{2},...,w_{n}\}$. 

We will define a list $x_{1}, x_{2},...,x_{\ell mn}$ of the vertices of $K_{\ell}\square K_{m}\square K_{n}$, organized as $\lcm(\ell,m,n) \times 3$ matrices, with the $k^\text{th}$ matrix denoted $A^{(k)}$.  We will define a total of $\frac{\ell mn}{\lcm(\ell ,m,n)}$ matrices.  The rows of the matrices produce the list of vertices in the natural way, with the rows of each matrix contributing the next $\lcm(\ell,m,n)$ vertices of the list, in order.  Precisely, if $A^{(k)}=\left[ a^{(k)}_{i,j} \right]$, and if $h=\lcm(\ell ,m,n)\cdot b +c$ where $c \in \{1,2,...,\lcm(\ell ,m,n)\}$, then $x_{h}$ is $\left(a^{(b+1)}_{c,1}, a^{(b+1)}_{c,2}, a^{(b+1)}_{c,3}\right)$.

The first matrix $A^{(1)}$ is defined as 
\begin{equation}\label{A1}
A^{(1)}=
\left[ {\begin{array}{ccc}
u_1 & v_1 & w_1\\
\rho (u_1) & \sigma (v_1) & \tau (w_1) \\
\rho ^2(u_1) & \sigma ^2 (v_1) & \tau ^2 (w_1)\\
\vdots & \vdots & \vdots \\
\rho ^{\lcm(\ell,m,n)-1}(u_1)& \sigma ^{\lcm(\ell,m,n)-1}(v_1)& \tau ^{\lcm(\ell,m,n)-1}(w_1)\\
\end{array}}\right],
\end{equation}
where $\rho \in S_{V(K_{\ell})}$ is the $\ell$-cycle $(u_{1}\ u_{2}\ \cdots \ u_{\ell})$, $\sigma \in S_{V(K_{m})}$ is the $m$-cycle $(v_{1}\ v_{2}\ \cdots \ v_{m})$, and $\tau \in S_{V(K_{n})}$ is the $n$-cycle $(w_{1}\ w_{2}\ \cdots \ w_{n})$.
We will find it helpful to think of the matrices in terms of their columns, so let 
$A^{(1)} = 
\left[ \begin{array}{ccc}
\mathbf{c}^{(1)} & \mathbf{d}^{(1)} & \mathbf{e}^{(1)}\\ 
\end{array} \right]$.
 For $1<k\leq \frac{\ell m n}{\lcm(\ell ,m,n)}$, let
 \begin{equation}\label{Ak}
A^{(k)}=\left[\begin{array}{ccc}
\mathbf{c}^{(k)}&\mathbf{d}^{(k)}&\mathbf{e}^{(k)}\\
\end{array}\right]=
\begin{cases}
\left[\begin{array}{ccc}
\mathbf{c}^{(1)}&\sigma\left(\mathbf{d}^{(k-1)}\right)&\mathbf{e}^{(k-1)}\\
\end{array}\right]&\text{if }k\equiv 1\ \left({\rm{mod}}\ \lambda\right)\\\\
\left[\begin{array}{ccc}
\mathbf{c}^{(1)}&\mathbf{d}^{(k-1)}&\tau\left(\mathbf{e}^{(k-1)}\right)\\
\end{array}\right]&\text{otherwise}
\end{cases} 
\end{equation}
where $\lambda=\frac{n\cdot\lcm(\ell ,m)}{\lcm(\ell ,m,n)}$.  Notice that the first columns of all $\frac{\ell m n}{\lcm(\ell ,m,n)}$ matrices are identical. See Table \ref{listexample} for an example of the list for $K_{3} \square K_{3} \square K_{6}$.

\begin{table}
\begin{tabular}{|c|c|c|}
\hline
$x_{1}=(u_{1}, v_{1}, w_{1})$&$x_{19}=(u_{1}, v_{2}, w_{3})$&$x_{37}=(u_{1}, v_{3}, w_{5})$\\
$x_{2}=(u_{2}, v_{2}, w_{2})$&$x_{20}=(u_{2}, v_{3}, w_{4})$&$x_{38}=(u_{2}, v_{1}, w_{6})$\\
$x_{3}=(u_{3}, v_{3}, w_{3})$&$x_{21}=(u_{3}, v_{1}, w_{5})$&$x_{39}=(u_{3}, v_{2}, w_{1})$\\
$x_{4}=(u_{1}, v_{1}, w_{4})$&$x_{22}=(u_{1}, v_{2}, w_{6})$&$x_{40}=(u_{1}, v_{3}, w_{2})$\\
$x_{5}=(u_{2}, v_{2}, w_{5})$&$x_{23}=(u_{2}, v_{3}, w_{1})$&$x_{41}=(u_{2}, v_{1}, w_{3})$\\
$x_{6}=(u_{3}, v_{3}, w_{6})$&$x_{24}=(u_{3}, v_{1}, w_{2})$&$x_{42}=(u_{3}, v_{2}, w_{4})$\\
\hline
$x_{7}=(u_{1}, v_{1}, w_{2})$&$x_{25}=(u_{1}, v_{2}, w_{4})$&$x_{43}=(u_{1}, v_{3}, w_{6})$\\
$x_{8}=(u_{2}, v_{2}, w_{3})$&$x_{26}=(u_{2}, v_{3}, w_{5})$&$x_{44}=(u_{2}, v_{1}, w_{1})$\\
$x_{9}=(u_{3}, v_{3}, w_{4})$&$x_{27}=(u_{3}, v_{1}, w_{6})$&$x_{45}=(u_{3}, v_{2}, w_{2})$\\
$x_{10}=(u_{1}, v_{1}, w_{5})$&$x_{28}=(u_{1}, v_{2}, w_{1})$&$x_{46}=(u_{1}, v_{3}, w_{3})$\\
$x_{11}=(u_{2}, v_{2}, w_{6})$&$x_{29}=(u_{2}, v_{3}, w_{2})$&$x_{47}=(u_{2}, v_{1}, w_{4})$\\
$x_{12}=(u_{3}, v_{3}, w_{1})$&$x_{30}=(u_{3}, v_{1}, w_{3})$&$x_{48}=(u_{3}, v_{2}, w_{5})$\\
\hline
$x_{13}=(u_{1}, v_{1}, w_{3})$&$x_{31}=(u_{1}, v_{2}, w_{5})$&$x_{49}=(u_{1}, v_{3}, w_{1})$\\
$x_{14}=(u_{2}, v_{2}, w_{4})$&$x_{32}=(u_{2}, v_{3}, w_{6})$&$x_{50}=(u_{2}, v_{1}, w_{2})$\\
$x_{15}=(u_{3}, v_{3}, w_{5})$&$x_{33}=(u_{3}, v_{1}, w_{1})$&$x_{51}=(u_{3}, v_{2}, w_{3})$\\
$x_{16}=(u_{1}, v_{1}, w_{6})$&$x_{34}=(u_{1}, v_{2}, w_{2})$&$x_{52}=(u_{1}, v_{3}, w_{4})$\\
$x_{17}=(u_{2}, v_{2}, w_{1})$&$x_{35}=(u_{2}, v_{3}, w_{3})$&$x_{53}=(u_{2}, v_{1}, w_{5})$\\
$x_{18}=(u_{3}, v_{3}, w_{2})$&$x_{36}=(u_{3}, v_{1}, w_{4})$&$x_{54}=(u_{3}, v_{2}, w_{6})$\\
\hline
\end{tabular}
\caption{The list of $x_{1}, x_{2},...,x_{54}$ for $K_{3} \square K_{3} \square K_{6}$ }
\label{listexample}
\end{table}
\end{subsection} 
\begin{subsection}{The list is an ordering of $V(K_{\ell} \square K_{m} \square K_{n})$}\label{ordering_section}
In this section we will prove that our list $x_{1}, x_{2},...,x_{\ell mn}$ is an ordering for $V(K_{\ell}\square K_{m}\square K_{n})$ by proving that it is in one-to-one correspondence with $V(K_\ell\square K_m\square K_n)$.  Since $|V(K_\ell\square K_m\square K_n)|=\ell mn$, we need only to prove that $x_i\neq x_j$ for all distinct $i,j\in\{1,2,\dots,\ell mn\}$.

Each matrix $A^{(k)}$ inherits a cyclical structure from $A^{(1)}$.  That is,
\begin{equation}\label{cyclical}
A^{(k)}=
\left[ {\begin{array}{ccc}
u_1 & v_i & w_j\\
\rho (u_1) & \sigma (v_i) & \tau (w_j) \\
\rho ^2(u_1) & \sigma ^2 (v_i) & \tau ^2 (w_j)\\
\vdots & \vdots & \vdots \\
\rho ^{\lcm(\ell,m,n)-1}(u_1)& \sigma ^{\lcm(\ell,m,n)-1}(v_i)& \tau ^{\lcm(\ell,m,n)-1}(w_j)\\
\end{array}}\right]
\end{equation}
for some $i\in\{1,2,\dots,m\}$, $j\in\{1,2,\dots,n\}$.  This structure gives us our first two steps in showing our list has no repetition.

\begin{prop}\label{norepinAk}
For any $k\in\left\{1,2,\dots,\frac{\ell mn}{\lcm(\ell,m,n)}\right\}$, the rows of $A^{(k)}$ are distinct.
\end{prop}
\begin{proof}
Consider the representation of $A^{(k)}$ given in \eqref{cyclical}.  In search of contradiction, suppose two rows in this matrix are identical.  Then there exist distinct $\alpha,\beta\in\{0,1,\dots,\lcm(\ell,m,n)-1\}$ such that $\rho^\alpha(u_1)=\rho^\beta(u_1)$, $\sigma^\alpha(v_i)=\sigma^\beta(v_i)$, and $\tau^\alpha(w_j)=\tau^\beta(w_j)$.  These respectively imply that $\alpha\equiv_\ell\beta$, $\alpha\equiv_m\beta$, and $\alpha\equiv_n\beta$, which in turn implies $\alpha\equiv_{\lcm(\ell,m,n)}\beta$.  However, as $\alpha$ and $\beta$ are distinct elements of $\{0,1,\dots, \lcm(\ell,m,n)-1\}$, this is not possible.  Therefore, no pair of identical rows exist.
\end{proof}

Because of the structure shown in \eqref{cyclical}, any row of $A^{(k)}$ determines the entire matrix.  And, because there are $\lcm(\ell,m,n)$ rows, if two matrices $A^{(k_1)}$ and $A^{(k_2)}$ share a common row, then they must share all rows (possibly cyclically permuted).  This gives us the following proposition.

\begin{prop}\label{samerows}
Let $k_1,k_2\in\left\{1,2,\dots,\frac{\ell mn}{\lcm(\ell,m,n)}\right\}$.  If there exists a row of $A^{(k_1)}$ that is also a row of $A^{(k_2)}$, then each row of $A^{(k_1)}$ is also a row of $A^{(k_2)}$.
\end{prop}

We can think of the first row as \emph{producing} the rest of the matrix; in view of this, we make the following definition.

\begin{definition}
 A vertex of $K_\ell\square K_m\square K_n$ is called a \emph{seed} if it corresponds to the first row of $A^{(k)}$ for some $k\in\left\{1,2,\dots,\frac{\ell mn}{\lcm(\ell,m,n)}\right\}$.
\end{definition}

In pursuit of proving that the list defined in \ref{definition} has no repeated vertices, we will make several observations about seeds.  From the definition of the list, given in \eqref{A1} and \eqref{Ak}, the first row of $A^{(1)}$ is $(u_1,v_1,w_1)$, and the first entry of the first row of $A^{(k)}$ is $u_1$ for all $k\in\left\{1,2,\dots,\frac{\ell mn}{\lcm(\ell,m,n)}\right\}$.  According to definition \eqref{Ak}, a matrix $A^{(k)}$ differs from its predecessor $A^{(k-1)}$ by an application of $\tau$ in the third column of $A^{(k-1)}$, unless $k\equiv 1\ ({\rm{mod}}\ \lambda)$.  In this case, $\sigma$
 instead is applied to the second column of $A^{(k-1)}$ to produce $A^{(k)}$.  This gives the pattern of the first rows, or seeds, given in Table \ref{firstrows}.  A new row of the table starts each time $k\equiv1\ ({\rm{mod}}\ \lambda)$; hence, the table has $\lambda$ columns.  Recall that the total number of matrices is $\frac{\ell mn}{\lcm(\ell,m,n)}$ and $\lambda=\frac{n\cdot\lcm(\ell,m)}{\lcm(\ell,m,n)}$; then $\frac{\ell mn}{\lcm(\ell,m,n)}\cdot\frac{1}{\lambda}=\gcd(\ell,m)$.  Table \ref{firstrows}, therefore, represents a total of $\gcd(\ell,m)$ rows.  Rows in the table are indexed by $i$, and we use $\gamma$ in the table to mean $\gcd(\ell,m)$.

{
\begin{table}[h]
\resizebox{\textwidth}{!}{
\begin{tabular}{|c|c|c|c|}
\hline
$k=1$&$k=2$
&&$k=\lambda$\\
$\left(u_1,v_1,w_1\right)$&$\left(u_1,v_1,\tau(w_1)\right)$
&$\cdots$&$\left(u_1,v_1,\tau^{\lambda-1}(w_1)\right)$\\
\hline
$k=\lambda+1$&$k=\lambda+2$
&&$k=2\lambda$\\
$\left(u_1,\sigma(v_1),\tau^{\lambda-1}(w_1)\right)$&$\left(u_1,\sigma(v_1),\tau^{(\lambda-1)+1}(w_1)\right)$
&$\cdots$&$\left(u_1,\sigma(v_1),\tau^{2(\lambda-1)}(w_1)\right)$\\
\hline
$k=2\lambda+1$&$k=2\lambda+2$
&&$k=3\lambda$\\
$\left(u_1,\sigma^2(v_1),\tau^{2(\lambda-1)}(w_1)\right)$&$\left(u_1,\sigma^2(v_1),\tau^{2(\lambda-1)+1}(w_1)\right)$
&$\cdots$&$\left(u_1,\sigma^2(v_1),\tau^{3(\lambda-1)}(w_1)\right)$\\
\hline
$\vdots$&$\vdots$
&$\cdots$&$\vdots$\\
\hline
$k=(i-1)\lambda+1$&$k=(i-1)\lambda+2$
&&$k=i\lambda$\\
$\left(u_1,\sigma^{i-1}(v_1),\tau^{(i-1)(\lambda-1)}(w_1)\right)$&$\left(u_1,\sigma^{i-1}(v_1),\tau^{(i-1)(\lambda-1)+1}(w_1)\right)$
&$\cdots$&$\left(u_1,\sigma^{i-1}(v_1),\tau^{i(\lambda-1)}(w_1)\right)$\\
\hline
$\vdots$&$\vdots$
&$\cdots$&$\vdots$\\
\hline
$k=(\gamma-1)\lambda+1$&$k=(\gamma-1)\lambda+2$
&&$k=\gamma\lambda$\\
$\left(u_1,\sigma^{\gamma-1}(v_1),\tau^{(\gamma-1)(\lambda-1)}(w_1)\right)$&$\left(u_1,\sigma^{\gamma-1}(v_1),\tau^{(\gamma-1)(\lambda-1)+1}(w_1)\right)$
&$\cdots$&$\left(u_1,\sigma^{\gamma-1}(v_1),\tau^{\gamma(\lambda-1)}(w_1)\right)$\\
\hline
\end{tabular}}
\caption{First rows of $\left\{A^{(k)}\right\}$, corresponding to the seeds of $K_\ell\square K_m\square K_n$, with $\gamma=\gcd(\ell,m)$}
\label{firstrows}
\end{table}
}

The seed in the $i^\text{th}$ row and $j^\text{th}$ column of Table \ref{firstrows} is given by
\begin{equation*}\label{sij}
\left(u_1,\sigma^{i-1}(v_1),\tau^{(i-1)(\lambda-1)+j-1}(w_1)\right)\end{equation*}
where $i\in\{1,2,\dots,\gcd(\ell,m)\}$ and $j\in\{1,2,\dots,\lambda\}$.  Because $i-1<\gcd(\ell,m)\leq m$, we can simplify the second {\color{black}component}, so the seed in the $i^\text{th}$ row and $j^\text{th}$ column of Table \ref{firstrows} is equal to
\begin{equation*}\label{sij}
\left(u_1,v_i,\tau^{(i-1)(\lambda-1)+j-1}(w_1)\right).
\end{equation*}
Each $k\in\left\{1,2,\dots,\frac{\ell mn}{\lcm(\ell,m,n)}\right\}$ is associated with a seed, as shown in Table \ref{firstrows}.  If we write $k=(b-1)\lambda+c$, with $c\in\{1,2,\dots,\lambda\}$, then the first row of $A^{(k)}$ is the entry of Table \ref{firstrows} in row $b$, column $c$, and we call this seed $s_k$:
\begin{equation}\label{sk}
s_k=\left(u_1,v_b,\tau^{(b-1)(\lambda-1)+c-1}(w_1)\right).
\end{equation}

Recall that our goal is to show that there is no repetition in our list of vertices.  These next facts we prove about seeds will allow us to do that.

\begin{prop}\label{seedsdistinct}
If $k_1,k_2\in\left\{1,2,\dots,\frac{\ell mn}{\lcm(\ell,m,n)}\right\}$, and $s_{k_1}=s_{k_2}$, then $k_1=k_2$.
\end{prop}
\begin{proof}
Let $k_1,k_2\in\left\{1,2,\dots,\frac{\ell mn}{\lcm(\ell,m,n)}\right\}$.  We can write $k_1=(b_1-1)\lambda+c_1$ and $k_2=(b_2-1)\lambda+c_2$ with $c_1,c_2\in\{1,2,\dots,\lambda\}$ and $b_1,b_2\in\{1,2,\dots,\gcd(\ell,m)\}$.  Suppose $s_{k_1}=s_{k_2}$.  It follows immediately from \eqref{sk} that $b_1=b_2$.  Also, it follows from $\tau^{(b_1-1)(\lambda-1)+c_1-1}(w_1)=\tau^{(b_2-1)(\lambda-1)+c_2-1}(w_1)$ that $c_1\equiv_n c_2$.
We know $c_1,c_2\in\{1,2,\dots,\lambda\}$.  Since $\lambda=\frac{n\cdot\lcm(\ell,m)}{\lcm(\ell,m,n)}\leq n$, we can get $c_1=c_2$.  We have shown $k_1=k_2$.
\end{proof}

\begin{prop}\label{noseedsametype}
Let $k_*\in\left\{1,2,\dots,\frac{\ell mn}{\lcm(\ell,m,n)}\right\}$, with $k_*=(b-1)\lambda+c$, $c\in\{1,2,\dots,\lambda\}$.  If $(u_1,v_b,w_z)$ is a row of $A^{(k_*)}$ other than the first row, then $(u_1,v_b,w_z)\neq s_k$ for any $k$.
\end{prop}
\begin{proof}
Let $k_*\in\left\{1,2,\dots,\frac{\ell mn}{\lcm(\ell,m,n)}\right\}$, with $k_*=(b-1)\lambda+c$, $c\in\{1,2,\dots,\lambda\}$; then the first row of $A^{(k_*)}$ is $s_{k_*}=(u_1,v_b,\tau^{(b-1)(\lambda-1)+c-1}(w_1))$ by \eqref{sk}.  If $(u_1,v_b,w_z)$ is any row of $A^{(k_*)}$ other than the first row, then 
$$(u_1,v_b,w_z)=(\rho^\gamma(u_1),\sigma^\gamma(v_b),\tau^{(b-1)(\lambda-1)+c-1+\gamma}(w_1))$$
where $\gamma\in\{1,2,\dots,\lcm(\ell,m,n)-1\}$ and $\gamma$ is an integer multiple of $\lcm(\ell,m)$.

In search of contradiction, suppose $(u_1,v_b,w_z)$ is a seed.  Since its second {\color{black}component} is $v_b$, we can see from Table \ref{firstrows} that $w_z=\tau^{(b-1)(\lambda-1)+c-1+\gamma}(w_1)\in\left\{\tau^{(b-1)(\lambda-1)+d}(w_1)\mid d\in\{0,1,\dots,\lambda-1\}\right\}$.  Then, for such $d$, $c+\gamma\equiv_n d+1$, or in other words, there exists an integer $e$ such that 
\begin{equation}\label{propstep}
ne-\gamma=c-(d+1).
\end{equation}
  Recalling that $\gamma$ is an integer multiple of $\lcm(\ell,m)$ and $\lambda=\frac{n\cdot\lcm(\ell,m)}{\lcm(\ell,m,n)}$, it is the case that $\lambda$ divides the lefthand side of \eqref{propstep}, and therefore $\lambda$ divides $c-(d+1)$.  Observing the constraints of constants $c$ and $d$, we see that $c-(d+1)\in\{-\lambda+1,-\lambda+2,\dots,\lambda-1\}$.  It follows that $c-(d+1)=0$.

Then equation \eqref{propstep} shows that $\gamma$ is not only an integer multiple of $\lcm(\ell,m)$, but an integer multiple of $\lcm(\ell,m,n)$.  However, as $\gamma\in\{1,2,\dots,\lcm(\ell,m,n)-1\}$, we have reached a contradiction.  Therefore, $(u_1,v_b,w_z)\neq s_k$ for any $k$.
\end{proof}

\begin{lemma}\label{lemma}
If $\left(u_1,v_i,w_j\right)$ is $s_k$, and $(u_1,v_y,w_z)$ is any row in $A^{(k)}$, then $y\equiv_{\gcd(\ell,m)} i$.
\end{lemma}
\begin{proof}
Let $\left(u_1,v_i,w_j\right)$ be the first row of $A^{(k)}$.  Then each row of $A^{(k)}$ takes the form $\left(\rho^\gamma(u_1),\sigma^\gamma(v_i),\tau^\gamma(w_z)\right)$.  If $(u_1,v_y,w_z)=\left(\rho^\gamma(u_1),\sigma^\gamma(v_i),\tau^\gamma(w_z)\right)$, then $\gamma$ is a multiple of $\ell$, say $\gamma=b\ell$.  And $y=i+b\ell\ ({\rm{mod}}\ m)$.  Then, for some integer $c$, $y=i+b\ell+cm$, and therefore $y\equiv i\ ({\rm{mod}}\ \gcd(\ell,m))$.
\end{proof}

\begin{prop}\label{norepseed}
Let $k_*\in\left\{1,2,\dots,\frac{\ell mn}{\lcm(\ell,m,n)}\right\}$.  If $(u_x,v_y,w_z)$ is any row of $A^{(k_*)}$ other than the first row, then $(u_x,v_y,w_z)\neq s_k$ for any $k$.
\end{prop}
\begin{proof}
In search of contradiction, let $(u_x,v_y,w_z)$ be a row of $A^{(k_*)}$ other than the first row, and suppose $(u_x,v_y,w_z)=s_k$ for some $k$.  Then necessarily $(u_x,v_y,w_z)=(u_1,v_y,w_z)$.  Take $s_{k_*}=(u_{1},v_{y'},w_{z'})$.  From Table \ref{firstrows}, $y,y'\in\{1,2,\dots,\gcd(\ell,m)\}$.  By Proposition \ref{noseedsametype}, $y\neq y'$.  And Lemma \ref{lemma} states that $y\equiv_{\gcd(\ell,m)}y'$.  But these three statements cannot be simultaneously true.  Therefore, $(u_x,v_y,w_z)\neq s_k$ for any $k$.
\end{proof}

\begin{prop}\label{norepinlist}
The list of vertices $x_1,x_2,\dots,x_{\ell mn}$ defined in Section \ref{definition} is pairwise distinct.
\end{prop}
\begin{proof}
In search of contradiction, suppose $x_i=x_j=(u_x,v_y,w_z)$ for distinct $i,j\in\{1,2,\dots,\ell mn\}$.  By Proposition \ref{norepinAk}, this means the vertex $(u_x,v_y,w_z)$ must appear as a row in two different matrices, call them $A^{(k_1)}$ and $A^{(k_2)}$, for some distinct $k_1,k_2\in\left\{1,2,\dots,\frac{\ell mn}{\lcm(\ell,m,n)}\right\}$.  Then, by Proposition \ref{samerows}, any row of $A^{(k_1)}$ is also a row of $A^{(k_2)}$.  So the first row of $A^{(k_1)}$, the seed $s_{k_1}$ is also a row of $A^{(k_2)}$.  However, $s_{k_1}$ cannot be the first row of $A^{(k_2)}$, as first rows of the matrices are distinct by Proposition \ref{seedsdistinct}.  Then $s_{k_1}$ must be some row other than the first row of $A^{(k_2)}$.  But this contradicts Proposition \ref{norepseed}.  Hence, $x_i\neq x_j$.
\end{proof}

Proposition \ref{norepinlist} shows that our list $x_1,x_2,\dots,x_{\ell mn}$ is in one-to-one correspondence with $V(K_\ell\square K_m\square K_n)$, achieving the goal of this section.
\begin{theorem}
The list of vertices $x_1,x_2,\dots,x_{\ell mn}$ defined in Section \ref{definition} is an ordering of the vertices of $K_\ell\square K_m\square K_n$.
\end{theorem}


\end{subsection}
\begin{subsection}{$K_{\ell}\square K_{m}\square K_{n}$ is radio graceful}

In this section we will show our ordering of $K_{\ell}\square K_{m}\square K_{n}$ induces a consecutive radio labeling.
\begin{theorem}
Let $\ell ,m,n \in \mathbb{Z}_+$, $\ell \leq m \leq n$, $\ell \geq 2$, $m,n \geq 3$ (excluding $K_2 \square K_3 \square K_3$). Then $K_{\ell}\square K_{m}\square K_{n}$ is radio graceful.
\end{theorem}
\begin{proof}
Let $\ell \leq m \leq n$, $\ell \geq 2$, $m,n \geq 3$ with either $\ell \geq 3$ or $n\geq 4$.  Also, let $x_1, x_2,...,x_{\ell mn}$ be the ordering of $V(K_{\ell}\square K_{m}\square K_{n})$ from Section \ref{ordering_section}.  Write $x_i = (u_i, v_i, w_i)$ and assume $x_i\in A^{(k)}$.  We will prove that our ordering satisfies the inequality \eqref{prelims_rgc} with $\Delta \in \{1,2\}$, which will finish the proof.

We begin with the case where $\Delta = 1$.  Note that $x_{i+1} = (\rho(u_i), \sigma(v_i), \tau(w_i))$ (if $x_{i+1}\in A^{(k)}$) and that $x_{i+1} \in \{(\rho(u_i), \sigma^2(v_i), \tau(w_i)), (\rho(u_i), \sigma(v_i), \tau^2(w_i))\}$ (if $x_{i+1}\in A^{(k+1)}$.)  Since $\ell\geq 2$, $\rho(u_i)\neq u_i$ and since $m,n\geq 3$, $v_i, \sigma(v_i),$ and $\sigma^2(v_i)$ are distinct, and similarly for $w_i$.  Thus, $x_i$ and $x_{i+1}$ always differ in all three coordinates, so $d(x_i, x_{i+1})) = 3$, satisfying the radio graceful condition of Proposition \ref{prelims_rgc_prop}.

We henceforth assume $\Delta = 2$.  Then $x_{i+2}$  is either $(\rho^2(u_i), \sigma^2(v_i), \tau^2(w_i))$ (if $x_{i+2}$ lies in $A^{(k)}$) or $x_{i+2}\in \{(\rho^2(u_i), \sigma^3(v_i), \tau^2(w_i)), (\rho^2(u_i), \sigma^2(v_i), \tau^3(w_i))\}$ (if $x_{i+2}\in A^{(k+1)}$).  We now break into cases depending on whether or not $\ell\geq 3$.

 Assume initially that $\ell\geq 3$.  Then the assumption that $\ell \leq m\leq n$ implies $m,n\geq 3$.  Then, $\rho^2(u_i)\neq u_i$, $\sigma^2(v_i)\neq v_i$, and $\tau^2(w_i)\neq w_i$.  Thus, $x_i$ and $x_{i+2}$ differ in at least two coordinates, so $d(x_i, x_{i+2})\geq 2$, satisfying the radio graceful condition of Proposition \ref{prelims_rgc_prop}.

Finally, assume $\ell = 2$, so $n\geq 4$.  This implies that $w_i, \tau^2(w_i)$,  and $\tau^3(w_i)$ are distinct.  If $m\geq 4$ as well, then $v_i, \sigma^2(v_i),$ and $\sigma^3(v_i)$ are distinct.  It follows in this case that $d(x_i, x_{i+2})\geq 2$.

The remaining case is when $\Delta = 2$, $\ell =2$, $m = 3$, and $n\geq 4$.  If $x_{i+2}$ lies in $A^{(k+1)}$, recall that $1\leq k+1 \leq \frac{\ell m n}{\lcm(\ell,m,n)} = \frac{6n}{\lcm(6,n)}$ and $\lambda = \frac{n\lcm(\ell, m)}{lcm(\ell,m,n)} = \frac{6n}{\lcm(6,n)} \geq k+1$.  Thus, we never satisfy that congruence $k+1 \cong 1\pmod{\lambda}$.  It follows that $x_{i+2} \neq (\rho^2(u_i), \sigma^3(v_i), \tau^2(w_i))$.  For the remaining two possibilities for $x_{i+2}$, we clearly have $d(x_i, x_{i+2})= 2$, satisfying the radio graceful condition of Proposition \ref{prelims_rgc_prop}.

\end{proof}
\end{subsection}

\begin{section}{Radio numbers in the exceptional cases}

In this section, we compute the radio numbers of $K_2\square K_3\square K_3$ and  $K_2\square K_2\square K_n$, beginning with  $K_2\square K_3\square K_3$.

To start, we note the ordering in Table \ref{table:sporadic} of the vertices of $K_2\square K_3\square K_3$ has a span of $20$.  Thus, $rn(K_2\square K_3\square K_3)\leq 20$.  We will later see that this ordering achieves the radio number of $K_2\square K_3\square K_3$.

\begin{center}

\begin{table}[h]

\begin{tabular}{|c|c||c|c||c|c|}

\hline

Vertex & Label & Vertex & Label & Vertex & Label\\

\hline

$(u_1,v_1,w_1)$ & $1$ & $(u_2,v_2,w_2)$ & $2$ & $(u_1,v_3,w_3)$ & $3$\\

$(u_2,v_1,w_1)$ & $4$ & $(u_1,v_2,w_2)$ & $5$ & $(u_2,v_3,w_3)$ & $6$ \\

$(u_1,v_1,w_2)$ & $8$ & $(u_2,v_2,w_3)$ & $9$ & $(u_1,v_3,w_1)$ & $10$ \\

$(u_2,v_1,w_2)$ & $11$ & $(u_1,v_2,w_3)$ & $12$ & $(u_2,v_3,w_1)$ & $13$\\

$(u_1,v_1,w_3)$ & $15$ & $(u_2,v_2,w_1)$ & $16$ & $(u_1,v_3,w_2)$ & $17$ \\

$(u_2,v_1,w_3)$ & $18$ & $(u_1,v_2,w_1)$ & $19$ & $(u_2,v_3,w_2)$ & $20$\\

\hline

\end{tabular}

\caption{A radio labeling of $K_2\square K_3\square K_3$}\label{table:sporadic}

\end{table}

\end{center}

\begin{prop} The radio number of $K_2\square K_3\square K_3$ is $20$.

\end{prop}

\begin{proof}  As we have already showed the radio number is at most $20$, we must now show $rn(K_2\square K_3\square K_3) \geq  20$.  To do this, consider the following claim $$ \text{$(\ast)$:  There is no consecutive radio labeling on any 7 vertices of $K_2\square K_3\square K_3$.}$$  Believing $(\ast)$, for any vertex labeling $x_1,...,x_{18}$ of $K_2\square K_3\square K_3$, there must be a jump in the labels among the vertices $x_1,...,x_7$ as well as among $x_{11},...,x_{18}$.  But if there are at least two jumps in the labels of the $18$ vertices, then the span must be at least $20$.

We now prove $(\ast)$.  Let $y_1,..., y_{7}$ be $7$ vertices in $K_2\square K_3\square K_3$ and assume for a contradiction that they can be consecutively radio labeled.  This implies that $d(y_i, y_{i+1}) = 3$ and $d(y_i, y_{i+2})\geq 2$.

Say $y_1 = (a,b,c)\in K_2\square K_3\square K_3$.  Then $y_2 = (a',b',c')$ where $a' \neq a$, $b'\neq b$, and $c'\neq c$.  Because $K_2$ only has two elements, $y_3 = (a, b'', c'')$.  Note that $d(y_3,y_2) =3$ implies $b''\neq b'$ and $c''\neq c'$.  Similarly, because the first coordinate of $y_1$ and $y_3$ match, the condition $d(y_1,y_3)\geq 2$ implies that $b''\neq b$ and $c''\neq c$.  Because $K_3$ only has three vertices in it, this means that $y_1$ and $y_2$ completely determine $y_3$.  Now, $y_4$ is determined in the same manner:  $y_4 = (a', b''',c''')$.  But the condition $b''\neq b'''\neq b'$ forces $b''' = b$, and similarly for $c$.  So $y_4 = (a', b, c)$.  Continuing, we find $y_5 = (a,b',c')$, $y_6 = (a',b'',c'')$, and $y_7 = (a,b,c) = y_1$.  Since $y_1\neq y_7$, this is a contradiction.

\end{proof}

We now turn our attention towards computing the radio number of $G_n:=K_2\square K_2\square K_n$.

\begin{prop}

The radio number of $G_n$ satisfies $rn(G_n)\geq 6n-1$.

\end{prop}

\begin{proof}

We first claim that no three vertices $y_1, y_2, y_3$ can have a consecutive labeling.  If $y_1$ is labeled $(a,b,c)$ and $y_2$ is labeled $(a',b',c')$, then $a\neq a'$, $b\neq b'$, and $c\neq c'$ because $d(y_1,y_2)$ must be equal to $3$ to have a consecutive labeling.  Likewise, $y_3$ is labeled $(a'',b'',c'')$ with $a''\neq a', b''\neq b'$ (which implies $a'' = a$ and $b'' = b$ since $K_2$  has two vertices).  But then $d(y_1,y_3) = 1$, so the labeling can not be consecutive.

Now, let $f:V_{G_n}\rightarrow \mathbb{Z}$ be a radio labeling, which is induced from an ordering $y_1, ..., y_{4n}$ of the vertices of $G_n$.  Since $f(y_{k+2}) - f(y_k) \geq 3$ for any $k$, and because $f(y_2)\geq 2$, we see \begin{align*} f(y_{4n}) &= (f(y_{4n}) - f(y_{4n-2})) + (f(y_{4n-2}) - f(y_{4n-4})) + ... + (f(y_4) - f(y_2)) + f(y_2) \\ &\geq 3(2n-1) + f(y_2) \\ &\geq 6n - 1.\end{align*}

Thus, $rn(G_n)\geq 6n-1$.

\end{proof}

Having established a lower bound for $rn(G_n)$, we now find an ordering whose span achieves this lower bound.

\begin{theorem}  Let $G_n = K_2\square K_2\square K_n$.  Then $rn(G) = 6n-1$.

\end{theorem}

\begin{proof}

By the previous proposition, we know $rn(G)\geq 6n-1$, so we need only find an ordering which has a span of $6n-1$.  First note that if $n = 1$, $K_2\square K_2\square K_1\cong K_2\square K_2$ has radio number $5 = 6(1)-1$ coming from the vertex ordering $(u_1,v_1),(u_2,v_2),(u_2,v_1),(u_1,v_2)$, which has labels $1,2,4,5$.  

For $G_2$, we use the ordering $$(u_1,v_1,w_1), (u_1,v_2,w_2), (u_2,v_1,w_1), (u_1,v_2,w_2), (u_2,v_1,w_2), (u_1,v_2,w_1), (u_1,v_1,w_2), (u_2,v_2,w_1).$$  This has labeling $$1,2,4,5,7,8,10,11 = 6(2)-1.$$  Notice that the last two vertices have the form $(u_1,v_1,w_n), (u_2,v_2,w_{n-1})$ with labels $6n-2, 6n-1$.

For $G_3$, we use the ordering $$(u_1,v_1,w_1), (u_2,v_2,w_2), (u_2,v_1,w_1), (v_1,v_2,w_2), (u_2,v_2,w_1), (u_1,v_1,w_3),$$ $$(u_1,v_2,w_1), (u_2,v_1,w_3), (u_1,v_1,w_2), (u_2,v_2,w_3), (u_2,v_1,w_2), (u_1,v_2,w_3)$$ which induces the labeling $$1,2,4,5,7,8,10,11,13,14,16,17 = 6(3)-1.$$  Notice that the last vertex has the form $(u_1,v_2,w_n)$, with label $6n-1$.

We find labelings for the remaining $G_n$ using induction, using both the $G_2$ and $G_3$ labelings as base cases.  For the induction hypothesis, we assume that when $n$ is even, we have found an ordering of the vertices of $G_n$ which ends with $(u_1,v_1,w_n), (u_2,v_2,w_{n-1})$ and with labels $6n-2$ and $6n-1$.  When $n$ is odd, we assume that we have found an ordering for the vertices of $G_n$ which ends with $(u_1,v_2,w_n)$ and label $6n-1$.

Then we order $G_{n+2}$ by copying the order on $G_{n}\subseteq G_{n+2}$ and then appending the remaining vertices in the order $$(u_1,v_1,w_{n+1}), (u_2,v_2,w_{n+2}), (u_2,v_1,w_{n+1}), (u_1,v_2,w_{n+2}), (u_2,v_1,w_{n+2}), (u_1, v_2, w_{n+1}), ( u_1,v_1, w_{n+2}), (u_2,v_2, w_{n+1}).$$  The corresponding labels are then $$6n+1, 6n+2, 6n+4, 6n+5, 6n+7, 6n+8, 6n+10, 6n+11 = 6(n+2) - 1.$$

\end{proof}

\end{section}


\begin{thebibliography}{99}

\bibitem{CEHZ} Gary Chartrand, David Erwin, Ping Zhang, and Frank Harary, {\it Radio labelings of graphs}, Bull. Inst. Combin. Appl. 33 (2001), 77--85.

\bibitem{Chartrand} Gary Chartrand and Ping Zhang, {\it Radio colorings of graphs---a survey}, Int. J. Comput. Appl. Math. 2 (2007), no. 3, 237--252.

\bibitem{GR92} Jerrold R. Griggs and Roger K. Yeh, {\it Labelling graphs with a condition at distance $2$}, SIAM J. Discrete Math. 5 (1992), no. 4, 586--595.

\bibitem{Hale} William K. Hale, {\it Frequency assignment: theory and applications}, Proceedings of the IEEE 68 (1980), no. 12, 1497--1514.

\bibitem{JLV} Jobby Jacob, Renu Laskar, and John Villalpando, {\it On the irreducible no-hole $L(2,1)$ coloring of bipartite graphs and Cartesian products}, J. Combin. Math. Combin. Comput. 78 (2011), 49--64.

\bibitem{kcp} Mustapha Kchikech, Riadh Khennoufa, and Olivier Togni, {\it Radio $k$-labelings for Cartesian products of graphs}, Discuss. Math. Graph Theory 28 (2008), no. 1, 165--178.

\bibitem{321prod} Byeong Moon Kim, Woonjae Hwang, and Byung Chul Song, {\it $L(3,2,1)$-labeling for the product of a complete graph and a cycle}, Taiwanese J. Math. 19 (2015), no. 3, 849--859.

\bibitem{prod} Byeong Moon Kim, Woonjae Hwang, and Byung Chul Song, {\it Radio number for the product of a path and a complete graph}, J. Comb. Optim. 30 (2015), no. 1, 139--149.

\bibitem{Sarah} Sarah Locke and Amanda Niedzialomski, {\it $K_n\square P$ is radio graceful}, Matematiche (Catania) 73 (2018), no. 1, 127--137.

\bibitem{MTWY} Marc Morris-Rivera, Maggy Tomova, Cindy Wyels, and Aaron Yeager, {\it The radio number of $C_n\square C_n$}, Ars Combin. 120 (2015), 7--21.

\bibitem{N16} Amanda Niedzialomski, {\it Radio graceful Hamming graphs}, Discuss. Math. Graph Theory 36 (2016), no. 4, 1007--1020.

\bibitem{321} Zhen-dong Shao and Jia-zhuang Liu, {\it The $L(3,2,1)$-labeling problem on graphs}, Math. Appl. (Wuhan) 17 (2004), no. 4, 596--602.

\bibitem{SR} B. Sooryanarayana and Raghunath P., {\it Radio labeling of cube of a cycle}, Far East J. Appl. Math. 29 (2007), no. 1, 113--147.
   
\bibitem{MC} Cindy Wyels and Maggy Tomova, {\it Radio Labeling Cartesian Graph Products}, 9th Cologne-Twente Workshop on Graphs and Combinatorial Optimization, Cologne, Germany, May 25-27, 2010. Extended Abstracts (2010), 163--167.

\end{thebibliography}
\end{document}